\documentclass{mystyle}
\usepackage[utf8]{inputenc}
\usepackage{amsmath}
\usepackage{calc}
\usepackage[dvipsnames]{xcolor}
\usepackage{accents}
\usepackage[all]{xy}                      %
  
\CompileMatrices                            

\UseTips                                    

\input xypic
\usepackage[bookmarks=true]{hyperref}       

\usepackage{amssymb,latexsym,amsmath,amscd}
\usepackage{xspace}
\usepackage{color}
\usepackage{graphicx}
\usepackage{dsfont}

\reversemarginpar

\theoremstyle{plain}
\newtheorem{theorem}{Theorem}[section]
\newtheorem*{theorem*}{Theorem}
\newtheorem{proposition}[theorem]{Proposition}

\newtheorem{lemma}[theorem]{Lemma}

\theoremstyle{definition}

\newtheorem{remark}[theorem]{Remark}
\newtheorem{example}[theorem]{Example}

\newtheorem*{conclusion}{Conclusion}

\newcommand{\enm}[1]{\ensuremath{#1}}          %
\newcommand{\op}[1]{\operatorname{#1}}
\newcommand{\cal}[1]{\mathcal{#1}}

\newcommand{\dbtilde}[1]{\accentset{\approx}{#1}}

\newcommand{\NN}{\enm{\mathbb{N}}}

\renewcommand{\AA}{\enm{\mathbb{A}}}
\newcommand{\PP}{\enm{\mathbb{P}}}

\newcommand{\Bb}{\enm{\cal{B}}}

\newcommand{\Ii}{\enm{\cal{I}}}

\newcommand{\Ll}{\enm{\cal{L}}}
\newcommand{\Mm}{\enm{\cal{M}}}

\newcommand{\Oo}{\enm{\cal{O}}}

\newcommand{\Rr}{\enm{\cal{R}}}

\renewcommand{\phi}{\varphi}
\renewcommand{\theta}{\vartheta}
\renewcommand{\epsilon}{\varepsilon}

\newcommand{\Pic}{\op{Pic}}

\newcommand{\dashdownarrow}{\mathrel{\rotatebox[origin=t]{-270}{\reflectbox{$\dashrightarrow$}}}}

\renewcommand{\to}[1][]{\xrightarrow{\ #1\ }}

\renewcommand{\sl}{\textsl{sl}}

\newcommand{\old}[1]{}
\newcommand{\vni}{\vskip 8pt \noindent}

\newcommand{\w}{\widetilde}
\newcommand{\h}{\ensuremath{\mathcal}}
\newcommand{\HL}{\ensuremath{\mathcal{H}^\mathcal{L}_}}
\begin{document}
\title[On the Hilbert scheme of smooth curves in $\PP^5$]
{On the Hilbert scheme of smooth curves of degree $15$ and genus $14$ in $\PP^5$}

\thanks{The first named author is a member of GNSAGA of INdAM (Italy). The second named author was supported in part by National Research Foundation of South Korea 
 (2022R1I1A1A01055306).} 




\author[E. Ballico]{Edoardo Ballico}
\address{Dipartimento di Matematica, Universita` degli Studi di Trento\\
Via Sommarive 14, 38123 Povo, Italy}
\email{ballico@science.unitn.it\\orcid.org/0000-0002-1432-7413}

\author[C. Keem]{Changho Keem}
\address{
Department of Mathematics,
Seoul National University\\
Seoul 151-742,  
South Korea}
\email{ckeem1@gmail.com}

\thanks{}



\date{\today}
\maketitle

\begin{abstract}
We denote by $\mathcal{H}_{d,g,r}$ the Hilbert scheme of smooth curves, which is the union of components whose general point corresponds to a smooth irreducible and non-degenerate curve of degree $d$ and genus $g$ in $\PP^r$. 
In this article, we show that $\h{H}_{15,14,5}$ is non empty and reducible with two components of the expected dimension hence generically reduced. We also study the birationality of the moduli map up to projective motion
and several key properties such as  gonality of a general element as well as specifying smooth elements of each components.

\end{abstract}

\section{Introduction}

We denote by $\h{H}_{d,g,r}$ the Hilbert scheme of smooth curves of degree $d$ and genus $g$ in $\PP^r$.
$\HL{d,g,r}$ is the subscheme of $\h{H}_{d,g,r}$ consisting of components of $\h{H}_{d,g,r}$ whose general element is linearly normal.

In this paper we study a certain peculiar Hilbert scheme of smooth curves in $\PP^5$ of degree $d=15, g=14$. Specifically,  we determine the number of components and study further property of $\h{H}_{15,14,5}$ such as the gonality of smooth element in each component.

Determining the irreducibility of a given Hilbert scheme is rather a non-trivial task, which goes back to the era of Severi \cite{Sev} who asserted that the Hilbert scheme  $\h{H}_{d,g,r}$ is irreducible for 
those triples of $(d,g,r)$ in the range 
\vskip 8pt

(i) $d\ge g+r~~~~$ or 

\noindent
\vskip 8pt
\noindent
in the following  Brill-Noether range which is much wider

\vskip 12pt
(ii) $\rho(d,g,r):=g-(r+1)(g-d+r)\ge 0$.  

\bigskip
The assertion of Severi turns out to be true for $r=3, 4$ under the condition (i); cf. \cite{E1, E2}. It is  also known that $\h{H}_{d,g,3}$ is 
irreducible in an extended range $d\ge g$; cf. \cite{KKy1} and references therein. For $r=4$, irreducibility of $\HL{d,g,4}$ also holds in the range $d\ge g+1$ except for some sporadic small values of the genus $g$; cf. \cite{KK3, KKy2}.

For $r=5$ the irreducibility of $\h{H}_{d,g,5}$ is not known in the range $d\ge g+5$ conjectured by Severi; the best known result so far regarding the irreducibility of $\h{H}_{d,g,5}$ is the result of H. Iliev who showed  that
$\h{H}_{d,g,5}$ is irreducible whenever $d\ge \max\{\frac{11}{10}g+2,g+5\}$; cf.
\cite{I2}.

Shifting  our attention to the family of linearly normal curves,  it makes sense talking about the Hilbert scheme  of linearly normal curves $\HL{d,g,r}$
only if $g-d+r\ge 0$ by Riemann-Roch formula; otherwise $\HL{d,g,r}=\emptyset$.
For the case  $r=5$ and in the non-trivial range $d\le g+5$, $\HL{d,g,5}$ is better understood 
than  $\h{H}_{d,g,5}$ in general. Moreover when the degree of the curve $d$ is relatively high with respect to the genus $g$, $\HL{d,g,5}$ behaves in quite reasonable manner;

\begin{itemize}
\item[(a)] $\HL{g+5,g,5}\neq\emptyset$ and is irreducible; \cite[Theorem 2.1]{JPAA}. 
\item[(b)] $\HL{g+4,g,5}\neq\emptyset$ and is irreducible if and only if $g\ge 6$; \cite[Theorem 2.2]{JPAA}.
\item[(c)] $\HL{g+3,g,5}\neq\emptyset$ and is irreducible if and only if $g\ge8$; \cite[Theorem 2.3, Remark 2.4]{JPAA}.
\item[(d)] $\HL{g+2,g,5}\neq\emptyset $ if and only if  $g\ge 10$ and is irreducible if and only if $g\ge 10 ~\&~g\neq  10, 12$; cf. \cite[Remark 3.2 (ii), Prop 3.3, Theorem 3.4, Theorem 3.7]{lengthy}, 
 \cite[Theorem 2.5, Remark 2.6]{JPAA}.
\item[(e)] $\HL{g+1,g,5}\neq\emptyset$ if and only if $g\ge 12$ by \cite[Theorem 2.4]{speciality}. 
However the irreducibility of $\HL{g+1,g,5}$ is known only for very low genus $g$; if $g=12$,
$\HL{g+1,g,5}$ is reducible  and is irreducible if $g=13$;
cf. \cite[proof of Theorem 2.4 \& Theorem 3.4]{speciality}
 \end{itemize}

\vskip 8pt
In this paper we focus our attention on the next case $\HL{15,14,5}$ as an initial attempt for a better understanding of  $\HL{g+1,g,5}$ for $g\ge 14$. 
We also study the birationality of the moduli map up to projective motion (Proposition \ref{e6} and Proposition \ref{e21})
and several key properties such as  gonality of a general element of any component (Proposition \ref{e2} and Remark \ref{e31}) as well as identifying smooth elements of each components; cf.  {\it Conclusion} right before Example \ref{nocom}.
The main result of this paper is the following theorem. 

\begin{theorem} \label{main}$\h{H}_{15,14,5}$ is reducible with two components of the (same) expected dimension.
\end{theorem} 

Indeed the case $g=14$ on which we focus  is the first non-trivial 
case concerning the irreducibility of $\h{H}_{g+1,g,5}$. For curves with higher genus $g\ge 15$, virtually nothing is known about the 
irreducibility of $\h{H}_{g+1,g,5}$. However the main result of the paper suggests that 
the irreducibility of $\h{H}_{d,g,5}$ beyond the conjectured range $d\ge g+5$ does not 
hold  for $d$ not too much below $g+5$.

It is also worthwhile to mention that $\h{H}_{15,14,5}=\HL{15,14,5}$  in our particular case $(d,g,r)=(15,14,5)$; every smooth curve in $\PP^5$ of degree $d=15$ and genus $g=14$ is linearly normal by Castelnuovo genus bound, i.e. genus $g=14$ is too large for a non-degenerate curve of degree $d=15$ sitting inside $\PP^6$.

The organization of this paper is as follows. In the next section, we prepare some basic preliminaries required for 
our study. We also determine the lower bound of the gonality of a smooth curve in $\h{H}_{15,14,5}.$
In the subsequent section we consider smooth curves $X\in\h{H}_{15,14,5}$ whose {\bf dual curves} $C\subset\PP^3$ -- by definition the image curve of a morphism induced by the residual series $K_X(-1)$ -- lie on quadric surfaces and identify one of the possible components of $\h{H}_{15,14,5}$. In section 4, we consider $X\in \h{H}_{15,14,5}$ whose dual curve
does not lie on a quadric surface and identify further possible component of $\h{H}_{15,14,5}$. 
In the last two sections we study the gonality of some/all curves and prove the uniqueness of a complete very ample linear series $g^5_{15}$ for a general curve  in each component of $\h{H}_{15,14,5}$ (Propositions \ref{e6} and \ref{e21}). One component is formed by $5$-gonal curves (Remark \ref{e31}).  The general element of the other component is $7$-gonal (Proposition \ref{e2}). No element of $\h{H}_{15,14,5}$ has gonality $\le 4$ (Propositions \ref{a1} and \ref{b1}).

For notation and conventions, we  follow those in \cite{ACGH} and \cite{ACGH2}; e.g. $\pi (d,r)$ is the maximal possible arithmetic genus of an irreducible,  non-degenerate and reduced curve of degree $d$ in $\PP^r$ which is usually referred as the first Castelnuovo genus bound. $\pi_1(d,r)$ is the so-called the second Castelnuvo genus bound which is the maximal possible arithmetic genus of  an irreducible, non-degenerate and reduced curve of degree $d$ in $\PP^r$ not lying on a  surface of minimal degree $r-1$; cf. \cite[page 99]{he}, \cite[page 123]{ACGH}.

Following classical terminology, a linear series of degree $d$ and dimension $r$ on a smooth curve $C$ is denoted by $g^r_d$.
A base-point-free linear series $g^r_d$ ($r\ge 2$) on a smooth curve $C$ is called {\bf birationally very ample} when the morphism 
$C \rightarrow \mathbb{P}^r$ induced by  the $g^r_d$ is generically one-to-one onto (or is birational to) its image curve.  A base-point-free linear series $g^r_d$ on $C$  is said to be compounded of an involution ({\bf compounded} for short) if the morphism induced by the linear series gives rise to a non-trivial covering map $C\rightarrow C'$ of degree $k\ge 2$. 
Throughout we work exclusively over the field of complex numbers.

\section{Some generalities and easy remarks }
\begin{remark}\label{new1}
\begin{itemize}
\item[(i)] Let $\Gamma$ be any irreducible component of the Hilbert scheme $\h{H}_{d,g,r}$ of smooth curves of degree $d$ and genus $g$ in  $\PP^r$. We recall that $$\dim \Gamma \ge 3g-3+\rho(d,g,r)+\dim\mathrm{Aut}(\PP^r)=(r+1)d+(3-r)(g-1).$$ In particular each irreducible component
of $\h{H}_{15,14,5}$ has dimension at least $64$.  Thus  an irreducible family $\h{F} \subset \h{H}_{15,14,5}$ such that  $\dim\h{F} \le 63$ is always contained in a larger family.    

\item[(ii)] Let $\xi$ be the natural functorial map $\xi: \h{H}_{15,14,5}\rightarrow\h{M}_{14}$. Let $\Gamma\subset  \h{H}_{15,14,5}$ be an irreducible component. Let  $\h{G}_\Gamma$ be the  irreducible family consisting of pairs 
$(g^3_{11}, p)$;  $g^3_{11}=|K_X(-1)|$, $X\subset\PP^5$ corresponds to $\xi^{-1}(p)$, $p\in\mathrm{Image ~}\xi(\Gamma)\subset\h{M}_{14}$.

\item[(iii)]
To each irreducible component $\Gamma\subset\h{H}_{15,14,5}$ (hence $\dim\Gamma\ge 64$), there is 
an irreducible family $\h{G}_\Gamma$ such that $$\dim\h{G}_\Gamma=\dim\Gamma-\dim\mathrm{Aut}(\PP^5)\ge 29.$$ In other words, if an irreducible family $\h{H}\subset\h{H}_{15,14,5}$ up to projective motion of $\PP^5$ followed by residualization 
of the hyperplane series (which is one-to-one) produces a family of linear series $g^3_{11}$ of dimension strictly less than $29$, then such a 
family $\h{H}$ does not constitute a full irreducible component. 

\end{itemize}
\end{remark}
For the existence part of particular components,  we utilize Example \ref{ex1} and Example \ref{ex2} and then use  Remark \ref{new1} to see that there are exactly two irreducible components of $\h{H}_{15,14,5}$ having the same expected dimension $64$ and hence the first cohomology of the normal bundle of $X$ vanishes for a general element $X$ in each of the $2$ irreducible components.

Fix a smooth $X\in \h{H}_{15,14,5}$ and let $B$ be the base locus of $|K_X(-1)|$. Set $b:= \deg (B)$. By Riemann-Roch $|K_X(-1)(-B)|$ induces a morphism $u_X: X\to \PP^3$. We say that $C:= u_X(X)$ is the {\bf dual curve}{\footnote{Readers are advised not to be confused with  another notion of ``dual curve" in the sense of Pl\"ucker.} of $X$.
Instead of working directly with a family in $\h{H}_{15,14,5}$, we will be working  with the family of (possibly singular) curves in $\PP^3$ which are the family of  dual curves $\{u_X(X); X\in\h{H}_{15,14,5}\}$. In fact, as we will see, this turns out to be more effective and easier to deal with; $\deg u_X(X)<\deg X$, the dimension of the ambient projective space gets lower, etc.

Note that $\deg (C)=\deg u_X(X)=11-\deg B$ if $u_X$ is birational onto its image.
We need the following key result concerning a family of nodal curves on a Hirzebruch surface $\mathbb{F}_e$; we only need it for $e=0$, the smooth quadric surface in $\PP^3$.

\begin{remark}\label{oo1}
\begin{itemize}
\item[(i)]
Take a Hirzebruch surface $\mathbb{F}_e$, $e\ge 0$ and  line bundle $L\cong \Oo_{\mathbb{F}_e}(ah+bf)$ on $\mathbb{F}_e$ such that $a>0$ and  $b\ge ae$; $h^2=-e, h\cdot f=1, f^2=0$.  A general element of $|L|$ is a smooth connected curve of genus   
$q=(a-1)(b-1-\frac{1}{2}ae)$.

\item[(ii)]
Fix an integer $g$ such that $0\le g\le q$. Let $\Sigma_g$ be the locus of all integral curves $Y\in |L|$ with geometric genus $g$. Let $\w{\Sigma}_g$ be the locus of all nodal $Y\in \Sigma_g$ with $\delta:=q-g$ nodes as its only singularities. By \cite[Cor. 4.4]{dedse} a general member of every irreducible component of  $\Sigma_g$ is a nodal curve. For a fixed $p\in\mathbb{F}_e$, being nodal at $p$ imposes three 
conditions on $|L|$ (or in any subvariety of $|L|$) and  hence the family of curves in $|L|$ with one node has codimension $3-\dim\mathbb{F}_e=1$ in $\Sigma_g$. Therefore $\w{\Sigma}_g$ has codimension $\delta$ in $|L|$ and for a general finite subset $S\subset \mathbb{F}_e$ with $\mathrm{Card}(S)=\delta$ there is an integral nodal curve $ T\in |L|$ with
$\mathrm{Sing} (T) =S$. 
By \cite{ty} the locus $\w{\Sigma}_g$ is irreducible of dimension $\dim |L| -\delta$.  Hence $\Sigma_g$ -- which is the closure of $\w{\Sigma}_g$  -- is also irreducible of dimension $\dim |L| -\delta$. 
\end{itemize}
\end{remark}

\vni The following inequality - known as Castelnuovo-Severi inequality - shall be used occasionally; cf. \cite[Theorem 3.5]{Accola1}.
\begin{remark}[Castelnuovo-Severi inequality]\label{CS} Let $C$ be a curve of genus $g$ which admits coverings onto curves $E_h$ and $E_q$ with respective genera $h$ and $q$ of degrees $m$ and $n$ such that these two coverings admit no common non-trivial factorization; if $m$ and $n$ are primes this will always be the case. Then
$$g\le mh+nq+(m-1)(n-1).$$ 
\end{remark}

The following simple fact regarding spanned linear systems on  smooth Del Pezzo surfaces
will be used  in the course of the proof of our main theorem. It should be noted that
the fact is widely known to people as a folklore. However the authors could not find an 
adequate source in the literature. 
\begin{remark}\label{2.3}
Let $S$ be a smooth Del Pezzo surface. A line bundle on $S$ is spanned (resp. very ample) if and only if it is nef (resp. ample) (\cite[Cor. 4.7]{sandra}).  Let $T\subset S$ be an integral projective curve such that $T^2>0$. Since $T$ is integral and $T^2>0$, $T\cdot D\ge 0$ for all curves $D$. Thus the line bundle $\Oo_S(T)$ is nef and hence it is spanned. Thus a general element of $|T|$ is smooth by the Bertini theorem.
\end{remark}

\begin{remark}\label{oo2}
We recall that -- as we have indicated in the introduction --  by Castelnuovo's upper bound for the arithmetic genus of an integral non-degenerate curve in $\PP^r$, $r\ge 6$ (\cite[Th. 3.13]{he}; take $d=15$, $r:= 6$ and hence $m=2$, $\epsilon =4$ to get $\pi(15,6) =13$), each $X\in \h{H}_{15,14,5}$ is linearly normal.
\end{remark}

\begin{remark}\label{oo3}
Being very ample is an open condition in any irreducible family of line bundles with prescribed degree and number of sections on a family of smooth curves of prescribed genus.
\end{remark}

\begin{proposition}\label{a1}
$\h{H}_{15,14,5}$ contains no trigonal curve.\end{proposition}
\begin{proof}
Assume the existence of a trigonal curve $X\in \h{H}_{15,14,5}$. Let $R$ be the trigonal line bundle on $X$ and let $m$ be the Maroni invariant of $X$, i.e. the first integer such that $h^0(R^{\otimes (m+2)}) \ge m+4$ (\cite[eq. 1.2]{ms}). By \cite[eq. 1.1]{ms},  $(14-4)/3<4\le m \le (14-2)/2=6$. Since $m\ge 4$, $h^0( R^{\otimes 5})=6$ and $R^{\otimes 5}$ is compounded. Since $\h{O}_X(1)$ is very ample, $\Oo_X(1) \ne R^{\otimes 5}$.

\begin{itemize}
\item[(a)]
 Assume that $u_X$ is induced by $|R|^{\otimes 3}$, i.e. that $b=2$ and $\deg (u_X) =3$. Thus $\Oo_X(1)\cong K_X\otimes (R^{\otimes 3})^\vee(-B)$ for some $B=p+p'\in X_2$.
Let $B'\in X_2$ be such that $p'+B'\in |R|$; $\Oo_X(1)(-B') \cong K_X\otimes (R^{\otimes 4})^\vee(-p)$.  Since $m\ge 4$, $h^0(R^{\otimes 4})=5$ and hence $h^0(K_X\otimes (R^{\otimes 4})^\vee)) =6$ and therefore $h^0(\Oo_X(1)(-B'))
= h^0(K_X\otimes (R^{\otimes 4})^\vee(-p))\ge 5$, contradicting the very ampleness of $\Oo_X(1)$.
\item[(b)]
Now assume that $u_X$ is not  induced by $g^1_3$. Thus the dual curve has degree $11-b$
and $K_X(-1)(-B)$ is not a multiple of  $g^1_3$, contradicting \cite[Prop. 1]{ms} and the fact that $\h{O}_X(1)$ is not a multiple of  $g^1_3$.
\end{itemize}
\end{proof}
\begin{lemma}\label{a2}
$u_X$ is birational onto its image.
\end{lemma}

\begin{proof}
Suppose $u_X$ is not birational onto its image.
Since $11$ is a prime and $C=u_X(X)$ spans $\PP^3$, we have the following two cases:
$$\begin{cases}
b=1, \deg (u_X)=2, \deg (C)=5\\

b=2, \deg (u_X)=3, \deg (C)=3; \mathrm{this ~case ~is ~excluded ~by ~Proposition ~\ref{a1}.}
\end{cases}$$
Now assume $\deg (u_X)=2$ and hence $\deg (C) =5$. 
Since $|K_X(-1)(-B)|$ is complete, $C$ has geometric genus $h=2$, smooth and $|\h{O}_C(1)|=g^3_5$ is non-special. 
Since $|K_X(-1)|=|u_X^*\h{O}_C(1)+B|$,
for any $p\in C$, we have 
\begin{align*}
|K_X(-1)+u_X^*(p)|&=|u_X^*(\h{O}_C(1))+B+u_X^*(p)|=|u_X^*(\h{O}_C(1)\otimes\h{O}_C(p))+B|\\&=|u_X^*(g^4_{6})+B|=g^{s}_{13}, ~~~s\ge 4
\end{align*}
hence 
\begin{align*}
|\h{O}_X(1)-u_X^*(p)|&=|K_X-K_X(-1)-u_X^*(p)|=|K_X- (K_X(-1)+u_X^*(p))|\\&=|K_X-g^s_{13}|=g^s_{13},
\end{align*}
and therefore $|\h{O}_X(1)|$ is compounded, a contradiction. 

\end{proof}

\begin{remark}\label{a3}
Assume  $b=\deg B>0$ and $\deg (u_X)=1$, i.e. assume $\deg (u_X(X))=11-\deg B\le 10$. Since $u_X(X)=C$ has geometric genus $g=14$ and $\pi_1(11-b,3)<g$, $C$ is contained in a unique quadric surface (\cite[Th. 3.13]{he}).
\end{remark}

\section{Dual curves contained in a quadric surface $Q\subset\PP^3$}
We assume that for $X\in\h{H}_{15,14,5}$ the dual curve $u_X(X)=C\subset\PP^3$ lies on a smooth quadric surface
$Q$.  Without loss of generality we may  assume that $C\in |\Oo_Q(a,11-b-a)|$ with $a \le 11-b-a$, i.e. $2a\le 11-b$. Proposition \ref{a1} implies $a\ge 4$ and hence $a=4$ if $b=2$ and $a\in \{4,5\}$ if $b\in \{0,1\}$.
Note that  $$\dim  |\Oo_Q(a,11-b-a)| =(a+1)(12-b-a) -1$$ and $$p_a(C) =(a-1)(10-b-a).$$ By Remark \ref{oo1},  in each case $a\in\{4,5\}$, we get an irreducible family whose general element is a nodal curve. We may also assume that the nodes are general in $Q$. Such a family has dimension
\begin{align*}\dim & |\Oo_Q(a,11-b-a)|-(p_a(C)-g)+\dim\PP(H^0(\PP^3,\h{O}(2))+b-\dim\textrm{Aut}(\PP^3)\\
&=(a+1)(12-b-a) -1-((a-1)(10-b-a)-14)+9+b-15\\
&=-b+29<29=\lambda(d,g,r)=3g-3+\rho(d,g,r); \textrm{ } (d,g,r)=(11,14,3),
\end{align*}
if $b>0$ 
and hence the family of such $X$'s does not constitute a full  component; Remark \ref{new1}(iii). 

Therefore we have $b=0$. 
Let $C\in|\Oo_Q(a,11-a)|$ with $\delta = p_a(C)-g$ nodes. Let $\tilde{C}\stackrel{\pi}{\rightarrow} C\subset Q$ be the normalization of the nodal curve $C$. In the following we check   if $\h{O}_{\tilde{C}}(1)=\pi^*(\h{O}_{{\PP^3}}(1))$
has very ample  residual series $|K_{\tilde{C}}(-1)|$; if this is the case, we then may conclude
that the family of nodal curves $C\subset Q\subset\PP^3$ under consideration indeed comes from a family of smooth curves $X\subset\PP^5$.

Before proceeding we recall some standard notations concerning linear systems and divisors on a blown up projective plane. Let $\PP^2_s$ be the rational surface $\PP^2$ blown up at $s$ general points. Let $e_i$ be the class of the exceptional divisor
$E_i$ and $l$ be the class of a line $L$ in $\PP^2$. For integers  $b_1\ge b_2\ge\cdots\ge b_s$, let $(a;b_1,\cdots, b_i, \cdots,b_s)$ denote class of the linear system $|aL-\sum b_i E_i|$ on $\PP^2_s$.  By abuse of notation we use the  expression $(a;b_1,\cdots, b_i, \cdots,b_s)$ for the divisor $aL-\sum b_i E_i$ and $|(a;b_1,\cdots, b_i, \cdots,b_s)|$ for the linear system $|aL-\sum b_i E_i|$. We use the convention $$(a;b_1^{s_1},\cdots,b_j^{s_j},\cdots,b_t^{s_t}), ~ \sum s_j=s$$ when  $b_j$ appears $s_j$ times consecutively  in the linear system $|aL-\sum b_i E_i|$.

\begin{example}\label{ex1}
\begin{itemize}
\item[(i)] $C\in  |\Oo_Q(5,6)|$:  $C$ is a nodal curve with $\delta=p_a(C)-g=6$ nodes. 
Choose a node $q_0\in C$ and blow up $Q$ at $q_0$ - which we denote by $Q_{q_0}$ - and then blow up successively at the remaining five
nodes $q_3,\cdots ,q_7$ to get  $Q_{q_0,q_3, \cdots q_7}$. Note that $Q_{q_0}\cong\PP^2_2$ the two point blow up of the projective plane $\PP^2$ and hence $Q_{q_0,q_3, \cdots q_7}\cong \PP^2_7$. Under the identification $Q_{q_0}\cong\PP^2_2$, the exceptional divisor of the blow up $Q_{q_0}\rightarrow Q$ is the proper transformation of the line through two points in $\PP^2$ which are the images of the two rulings of $Q$ under the projection $Q-\{{q_0}\}\rightarrow\PP^2$. 

Let $\tilde{C}\in |(d;b_1,\cdots ,b_7)|\in \Pic(\PP^2_7)$, where $\tilde{C}$ is the smooth curve after resolving all the nodes of $C$. Since $q_0\in C\subset Q$ as well as the the remaining $5$ points $q_3,\cdots ,q_7 $ are double points, we have
$$(l-e_1-e_2)\cdot \tilde{C}=d-b_1-b_2=2, ~~b_j=e_j\cdot\tilde{C}=2, ~~j=3,\cdots ,7.$$  Note that 
$$\deg C=(2l-e_1-e_2)\cdot \tilde{C}=2d-b_1 -b_2=11$$
$$\tilde{C}\cdot (l-e_1)=a=5, ~~\tilde{C}\cdot (l-e_2)=11-a=6$$ and therefore 
$\tilde{C}\in |(9;4,3,2^5)|$
from which it follows
\begin{align*}
|K_{\tilde{C}}(-1)|&=|K_{\PP^2_7}+\tilde{C}-(2l-e_1-e_2)|_{|\tilde{C}}\\
&=|-(3;1^7)+(9;4,3,2^5)-(2;1^2,0^5)|=|(4;2,1^6)|_{|\tilde{C}}.
\end{align*}
We note that the restriction map $$\h{M}:=|K_{\PP^2_7}+\tilde{C}-(2l-e_1-e_2)|\longrightarrow |K_{\PP^2_7}+\tilde{C}-(2l-e_1-e_2)|_{|\tilde{C}}$$ is surjective;  by Kodaira vanishing theorem, $$h^1(\PP^2_7, \h{I}_{\w{C}}\otimes\h{M})=h^1(\PP^2_7, -(5;2^2,1^5))=0$$  since
$|(5;2^2,1^5)|$ is (very) ample.
Hence $|K_{\tilde{C}}(-1)|$ is completely cut out by the very ample  linear system $|(4;2,1^6)|$ on $\PP^2_7$ (by \cite{sandra}) and the very ampleness of 
 $|K_{\tilde{C}}(-1)|$ follows.

\item[(ii)] $C\in  |\Oo_Q(4,7)|$:  $C$ is nodal with $\delta=p_a(C)-g=4$ nodes. We carry out 
a similar computation as in (i).
Choose a node $q_0\in C$ and blow up $Q$ at $q_0$ - which we again denote by $Q_{q_0}$ - and then blow up successively at the remaining three 
nodes $q_3,\cdots ,q_5$ to get  $Q_{q_0,q_3, \cdots q_5}$. 

Let $\tilde{C}\in |(d;b_1,\cdots ,b_5)|\in \Pic(\PP^2_5)$, where $\tilde{C}$ is the smooth curve after resolving all the $\delta$ nodes of $C$. Since $q_0\in C\subset Q$ and the remaining $3$ points $q_3,\cdots ,q_5 $ are double points, we again have 
$$(l-e_1-e_2)\cdot \tilde{C}=d-b_1-b_2=2 ~~\mathrm{and}~~b_j=e_j\cdot\tilde{C}=2, j=3,\cdots ,5.$$
Note that 
$$\deg C=(2l-e_1-e_2)\cdot \tilde{C}=2d-b_1-b_2=11$$ 
$$\tilde{C}\cdot (l-e_1)=4, ~~\tilde{C}\cdot (l-e_2)=7$$ and therefore
$\tilde{C}\in |(9;5,2^4)|.$

\vni
We set 
\begin{align*}\h{M}&:=|K_{\PP^2_5}+\tilde{C}-(2l-e_1-e_2)|\\
&=|-(3;1^5)+(9;5,2^4)-(2;1^2,0^3)|=|(4;3,0,1^3)|
\end{align*}
and consider the exact sequence
$$0\rightarrow \h{O}(-\tilde{C}+\h{M})\rightarrow\h{O}(\h{M})\rightarrow\h{O}(\h{M})_{|\tilde{C}}\rightarrow 0 $$
Note that 
$\h{L}:=\h{O}(-\tilde{C}+\h{M})=\h{O}(-(5;2^2,1^3))$ and $\h{L}^{-1}=\h{O}((5;2^2,1^3))$ is ample (indeed very ample) and hence
$$h^1(\PP^2_5,\h{O}(-\tilde{C}+\h{M}))=0$$
by Kodaira vanishing theorem. 
By the surjectivity of restriction map 
$$H^0(\PP^2_5,\h{O}_{\PP^2_5}(K_{\PP^2_5}+\tilde{C}-(2,1^2,0^3))\rightarrow H^0(\tilde{C}, K_{\tilde{C}}(-1)),$$
$|K_{\tilde{C}}(-1)|$ is completely cut out on $\tilde{C}$ by $|K_{\PP^2_5}+\tilde{C}-(2l-e_1-e_2)|$.

Note that the linear system $\h{M}=|(4;3,0,1^3)|$ contracts the exceptional divisor $e_3$,  whereas 
$e_3\cdot\tilde{C}=2$. Hence the morphism induced by $\h{M}$  on $\PP^2_5$ and then restricted to $\tilde{C}$ produces 
a singularity on the image curve.
In sum, we conclude that the normalization $\w{C}$ of $C\in  |\Oo_Q(4,7)|$ does not have a smooth counterpart in $\PP^5$ (or very ample residual series $|K_{\tilde{C}}(-1)|$)  and hence does not contribute to a component of $\h{H}_{15,14,5}$. 
\end{itemize}
\end{example}

\begin{remark} (i) If a dual curve of a smooth $X\in\h{H}_{15,14,5}$ lies on a quadric cone, such curves are in the boundary of the component corresponding to the family of curves $ C\in|\Oo_Q(5,6)|$ of geometric genus $g=14$ on  a smooth quadric surface $Q\subset\PP^3$;  recall that  curves on a quadric cone is specialization of curves on a
smooth quadric by \cite{Zeuthen}; see also \cite[Introduction]{Brevik}.

(ii) We take $|L|=|\Oo_Q(5,6)|$  in Remark \ref{oo1}. We have
$$\dim\Sigma_{14}=\dim|\Oo_Q(5,6)|+14-p_a(C)=35,$$
hence the irreducible  family of smooth curves in $\PP^5$ of degree $15$ and genus $g=14$ consisting of image curves of the morphism $\w{C}\hookrightarrow\PP^5$ induced by $|K_{\w{C}}(-1)|$  has dimension 
$$\dim\Sigma_{14}-\dim\mathrm{Aut}(Q)+\dim\mathrm{Aut}(\PP^5)=64,$$
which may well constitute a component of $\h{H}_{15,14,5}$.
\end{remark}


\section{The final step of the proof of Theorem \ref{main}}

In this section we finish the proof of Theoerm \ref{main}.
Now we consider the case in which $h^0(\PP^3,\Ii_C(2)) =0$. By Remark \ref{a3}, $b=0$ and $\deg (C)=11$. Recall that $C$ has geometric genus $g=14$ and hence $p_a(C) \ge 14$. 
\vskip 8pt
 (a) Let $C\subset\PP^3$ be non-degenerate curve not contained in a cubic surface. Since $\deg (C)> 3^2$ and by \cite{gp2}, a general hyperplane section $H\cap C=\Delta$ of $C$ is not contained in a plane cubic. 
One then easily computes 
$$h^1(\h{I}_{\Delta}(1)) =8, h^1(\h{I}_{\Delta}(2)) =5, h^1(\h{I}_{\Delta}(3)) =1,  h^1(\Ii_{\Delta}(t)) =0 \textrm{ for all }t\ge 4.$$
By  \cite[Cor. 3.2]{he} one knows 
$$ p_a(C)\le \sum d-h_\Delta(n)=\sum h^1(H,\h{I}_\Delta (n))$$ where $h_\Delta(n)$ is the Hilbert function of $\Delta$. Hence $p_a(C)=14$ and therefore $C$ is smooth. Furthermore $C$ is arithmetically Cohen-Macaulay (ACM for short)   by \cite[Remark 3.1.1]{he}.  In particular  $h^0(\Ii_C(4)) =h^0(\PP^3,\h{O}(4))-h^0(C,\h{O}_C(4))= 4$. 

By \cite[Prop. 3.1]{ps}, $C$ is directly linked to a ACM curve $D$ of degree $5$ and $p_a(D)=2$ by a complete intersection of two general quartics containing $C$ (i.e. a general pencil of quartics containing $C$ which may be regarded as a general member
of the Grassmannian $\mathbb{G}(1,\PP(H^0(C,\Ii_C(4))))$.

Conversely let $\h{E}$ be the family consisting of all ACM curves $D\subset\PP^3$ of degree $5$ and $p_a(D)=2$. It is known that $\h{E}$ is irreducible of dimension $4\cdot 5=20$ and a general element of $\h{E}$ is  smooth. 
For every $D\in \h{E}$, $$h^0(D,\Ii_D(4)) = h^0(\PP^3,\h{O}(4))-h^0(D,\h{O}(4)=35-(20-2+1)=16$$ and $D$ is directly linked to an ACM curve $C$ with $p_a(C)=14$ and degree $11$. 

Therefore we may consider the locus
$$\Sigma\subset\mathbb{G}(1,\PP(H^0(\PP^3,\h{O}(4))))=\mathbb{G}(1,34)$$
of pencils of quartic surfaces whose base locus consists of a curve $C$ of degree $d=11$ and genus $g=14$ and a quintic $D$ of genus $h=2$ where $C$ and $D$ are directly linked via a complete intersection of quartics corresponding to the pencil, together with  two obvious maps 
\[
\begin{array}{ccc}
\hskip -48pt\mathbb{G}(1,34)\supset\Sigma&\stackrel{\pi_C}\dashrightarrow
\quad\Gamma\subset \hskip 3pt\h{H}_{11,14,3}\\
\\
\dashdownarrow\vcenter{\rlap{$\scriptstyle{{\pi_D}}\,$}}
\\ 
\\
\h{E}=\h{H}_{5,2,3}^{\tiny{\textrm{ACM}}},
\end{array}.
\]
and the two dotted arrows are dominant maps.

As we have seen, $\pi_D$ is generically surjective with fibers open subsets of $$\mathbb{G}(1,\PP(H^0(D,\Ii_D(4))))=\mathbb{G}(1,15)$$
where $h^0(D,\Ii_D(4)) =16.$ Since $\h{E}$ is irreducible, it follows that  $\Sigma$ is irreducible and $$\dim\Sigma=\dim\mathbb{G}(1,15)+\dim\h{E}=28+4\cdot 5=48.$$ On the other hand, since every $C\in\Gamma$ lies on exactly $4$ independent quartics, $\pi_C$ is generically surjective with fibers open in $\mathbb{G}(1,\PP(H^0(\Ii_C(4)) )=\mathbb{G}(1,3)$. Finally it follows that the family $\Gamma$ is irreducible of dimension
$$\dim\Sigma-\dim\mathbb{G}(1,3)=44=4\cdot11, $$ 
which in turn implies that the corresponding family consisting of dual curves of $C\in\Gamma$ (i.e. curves induced by the residual series $|K_C(-1)|$) forms an
irreducible family of smooth curves in $\PP^5$ of degree $15$ and genus $g=14$ of dimension
$$\dim\Gamma-\dim\mathrm{Aut}(\PP^3)+dim\mathrm{Aut}(\PP^5)=64$$
under the condition that $|K_C(-1)|$ is very ample for a general $C\in\Gamma$. 
\vskip 8pt
In the  example below -- which is due to Angelo Lopez\footnote{The authors are grateful to Angelo Lopez for the example. } -- we will see that for a certain $C\in\Gamma$, $|K_C(-1)|$ is very ample and hence  is very ample for a general $C\in\Gamma$; cf. Remark \ref{oo3}.
\begin{example}\label{ex2}
We choose smooth $C\in\Gamma$ which lies on a smooth quartic surface $S$. 
Let $D$ be the quintic of genus $h=2$ directly linked to $C$ via complete
intersection of $S$ and another quartic $T$.  
Let $Q$ be a quadric (which may be chosen smooth) containing $D\in  |\h{O}_Q(2,3)|$ and let $E\in  |\h{O}_Q(2,1)|$ be the twisted cubic curve such that $D+E=S\cap Q\in |\h{O}_Q(4,4)|$.
Set $|H|=\mathcal{O}_S(1)$.

Note that $C\in |4H-D|$ and $D+E\in |2H|$. Since $S$ is a K3 surface, $K_S=0$ and hence
$$|K_C(-1)|=|K_S+C-H|_{|C}=|3H-D|_{|C}=|H+E|_{|C}$$

We claim that $|H+E|$ induces an embedding of $C$. Fix a degree $2$ zero-dimensional scheme $Z\subset C$. If $Z\cap E=\emptyset$ we just use that $E\cap C$ is effective and $|H|$ gives an embedding.
Now assume $Z\subset E\cap C$.
We consider the exact sequence 
\begin{equation}\label{eqbo1}0\rightarrow \h{O}_S(H)\rightarrow \h{O}_S(H+E)\rightarrow \h{O}_S(H+E)_{|E}\rightarrow 0.\end{equation}
We have $(H+E)\cdot E =1$ and hence $\Oo_S(H+E)_{|E}$ is a degree $1$ line bundle on $E\cong\PP^1$. Thus $\Oo_E(H+E)$ is very ample.
Since $h^1(\Oo_S(H))=0$,  we get that $Z$ imposes $2$ independent conditions on $|H+E|$ and hence on $|H+E|{_{|C}}$. Now assume $Z\cap E\ne \emptyset$ and $Z\nsubseteq E$. Write $Z\cap E=\{p\}$. The point $p$ is not a base point
of $\Oo_C(H+E)$, because $|H+E|_{|E}$ has no base point. Since $H$ is globally generated, there is $M\in |\Oo_S(H)|$ such that $Z\nsubseteq M$. Since $Z\nsubseteq E$, $Z$ imposes $2$ independent conditions on $|\Oo_C(H+E)|$.\end{example}
\quad (b) We assume that $C$ is contained in an irreducible cubic surface $S$. Since $\deg (C)=11>9$, such cubic surface is unique. Since $h^0(\PP^3,\h{I}_C(2)) =0$, we get $p_a(C)\le\pi_1(11,3)=15$.
Thus either $C$ is smooth or it has a unique singular point, which is either an ordinary node or an ordinary cusp. There are several possibilities for the cubic surface $S$ containing $C$.
\begin{itemize}
\item[(i)] $S$ is a cubic ruled surface which is projection of a rational normal surface scroll $S'\subset\PP^4$ from a point $p\notin S'$.
\item[(ii)] $S$ is a projection of a cone $S''\subset\PP^4$ over a twisted cubic curve in $\PP^3$ from $p\notin S''$, i.e. $S$ is a cone over a singular plane cubic which has a double line. 
\item[(iii)]$S$ is a cone over a non-singular plane curve.
\item[(iv)] $S$ is a non-singular cubic surface. 
\item[(v)] $S$ is  a singular and normal surface with isolated singularities.

\end{itemize}
We remark that the first two cases (i) and (ii) may be eliminated from our consideration just because 
the morphism $u_X: X\rightarrow S\subset\PP^3$ with $C$ as its image is induced by a complete linear system; note that the morphism $u_X$ lifts to a morphism $X\rightarrow S' (\mathrm{ ~~or~~ } S'' )\subset \PP^4$ then followed by an external projection into $\PP^3$ which cannot be induced by a complete linear series.

In the third case (iii) we need a dimension count as follows. Fix a smooth cubic curve $E\subset\PP^2$ and 
a triple covering $X\stackrel{\pi}{\longrightarrow} E$ such that $$|K_X(-1)|=|\pi^*(\h{O}_E(1))+B|=g^3_{11},  B\in X_2.$$
Choice of $E$  depends on $\dim\h{M}_1+\dim W^2_3(E)=2$ parameters.
Choice of a triple covering $X$ over  a fixed $E$ depends on $2g-3$ parameters by Riemann's moduli count; \cite [Theorem 8.23, p 828]{ACGH2}.
Choice of degree two effective divisors $B\in X_2$ depends on at most one 
parameter; note that $|\pi^*(\h{O}_E(1))+B|=g^2_{11}$ for general $B\in X_2$.
However these numbers do not add up the minimal possible dimension $\lambda(15,14,5)=29$ which is necessary to constitute
a component.

In the case (iv), we first assume $C=u_X(X)$ is smooth.  By \cite[Prop. B1]{gp1} the family of smooth curves $C\subset S$ of degree $d=11$ has dimension at most $d+g+18=43< 44$ and therefore does not constitute a full component. If the image curve $u_X(X)=C$ is singular (and hence $p_a(C)=15$), such $C$ forms a family of positive codimension inside a family generically consisting of smooth curves of degree $11$ and of genus $g=p_a(C)=15$, which forms a family of dimension $11+p_a(C)+18=44$; cf. Remark \ref{2.3}.





In the last case (v), note that $C=u_X(X)$ can be either singular or smooth. However, in both cases, the family 
consisting of curves on normal $S$ has dimension strictly less than $11+p_a(C)+18\le44$ by \cite{Brevik} hence does not constitute a full component. 

\begin{conclusion}
We now have exhausted all the possibilities for the dual curve $C\subset\PP^3$, therefore we conclude that 
there are only two irreducible components of $\h{H}_{15,14,5}$; 

(i) one component generically consisting of curves 
whose dual curve $C$ is a nodal curve in $|\h{O}_Q(5,6)|$ and 

(ii) the other component generically consisting of image curves of the morphism $C\hookrightarrow\PP^5$ induced by $|K_C(-1)|$ where $C\subset\PP^3$ is  directly linked to a quintic of genus $h=2$.

Since both families have the same dimension $64$, one is not in the boundary of the other. 
\end{conclusion}


\begin{example}\label{nocom}
There is a  singular curve $C\subset\PP^3$ with $(d,g)=(11,14)$ lying on a smooth cubic surface 
such that $|K_{\widetilde{C}}(-1)|$ is very ample, where $\w{C}\longrightarrow C$ is the normalization.  Therefore the image curve $X$ of $\w{C}$ under $\w{C}\stackrel{\tiny{|K_{\w{C}}(-1)|}}{\longrightarrow}X\subset\PP^5 $ is a curve with the right degree and genus $(d,g)=(15,14)$. However such curves does not constitute a full component as we remarked in the last stage of the proof of the main theorem. 
\vni
Let $S\subset\PP^3$ be a smooth cubic surface. Take
$C\in |(10;4,3^5)|$; $p_a(C)=15$ and assume that $C$ has one node and no further singularities. By blowing up the only node of $C$, the proper transformation of $C$ is $\w{C}\in|(10;4,3^5,2)|$ and 
$$|K_{\w{C}}(-1)|=|\w{C}+K_{\w{\PP^2_7}}-H|_{|\w{C}}=|(10;4,3^5,2)-(3;1^7)-(3,1^6,0)|=|(4;2,1^6)|$$
Note that $\dim|(4;2,1^6)|=14-3-6=5$ and $|(4;2,1^6)|$ is very ample on $\PP^2_7$ by \cite{sandra}.
Therefore $\w{C}\hookrightarrow X\subset\PP^5$ is an embedding with smooth image curve $X$ of 
degree $d=(10;4,3^5,2)\cdot(4,;2,1^6)=15$ and genus $g=p_a(\w{C})=p_a(C)-1=14$.
The dimension of the family of very ample $g^5_{15}$'s arising this way is
\begin{align*}
\dim|(10;4,3^5)|&-1+\dim\PP(H^0(\PP^3,\h{O}(3))-\dim\mathrm{Aut}(\PP^3)\\&=24+19-15=28<\lambda(15,14,5) =29,
\end{align*}
and therefore this family does not constitute a full component. 

By using Sakai's method \cite{Sakai} which is an effective numerical criteria for the gonality
of a plane curve with prescribed singularities, $\tilde{C}$ is $6$-gonal which is left to readers for  verification.  We also note that 
the curve $X\subset\PP^5$ lies on a surface of degree six since $(4;2,1^6)^2=6$.
\end{example}
\section{Gonality of a general element of a component of $\h{H}_{15,14,5}$}

Let $\Gamma_1$ (resp. $\Gamma _2$) be the irreducible component of $\h{H}_{15,14,5}$ whose general element has dual not contained (resp. contained) in a quadric surface.

\begin{proposition}\label{e2}
A general $X\in \Gamma _1$ is 7-gonal.
\end{proposition}

\begin{proof}
The dual curve $C$ of $X$ is a smooth degree $11$ ACM curve $C\subset \PP^3$ whose homogeneous ideal is generated by forms of degree $4$. Thus there is no line $J\subset\PP^3$ such that $\deg (J\cap C)\ge 5$.
There are lines $J$ such that $\deg (J\cap C)=4$ (\cite[Proposition 3.1]{hs}). Thus $C$ is 7-gonal (\cite[Theorem 1.3]{hs}).
\end{proof}


\begin{proposition}\label{b1}
No smooth element $X\in 
\h{H}_{15,14,5}$ is 4-gonal.
\end{proposition}

\begin{proof}
Recall that $\deg B\le 1$, where $B$ is the base locus of $|K_X(-1)|$.
We also recall that $u_X$ is birationally very ample or very ample by Lemma \ref{a2}.

We now assume the existence of a 4-gonal $X\in \h{H}_{15,14,5}$ with a unique $g^1_4$.
\begin{itemize}
\item[(i)] $\deg B=1$: $\phi$ is birationally very ample and by Remark \ref{a3}, $C$ is contained in a quadric $Q\subset\PP^3$. Assume that $Q$ is smooth.  If $C\in|\h{O}_Q(5,5)|$, $X$ has a base point free $g^1_5$ implying that $g\le (4-1)(5-1)=12$ by Castelnuovo-Severi inequality, a contradiction. 
If $C\in|\h{O}_Q(4,6)|$, $p_a(C)=15$ hence $C$ has a double point which is either a node or a cusp.
Let $q\in C\subset Q$ be the singular point and we blow up $Q$ at $q$ to get $Q_q\cong\PP^2_2$. Let $\w{C}\subset\PP^2_2$ be the proper transformation of $C$ and 
set $\w{C}\in|(d;b_1,b_2)|$. We carry a similar computation as we did before as follows. 
The exceptional divisor of the blow up is $l-e_1-e_2$. Since $q$ is a double point
$\w{C}\cdot (l-e_1-e_2)=d-b_1-b_2=2$. Also $\w{C}\cdot (2l-e_1-e_2)=2d-b_1-b_2=\deg C=10$,
$\w{C}\cdot (l-e_1)=d-b_1=4$,  $\w{C}\cdot (l-e_1)=d-b_2=6$
and hence we have $\w{C}\in (8;4,2)$.
We set 
\begin{align*}\h{M}&:=|K_{\PP^2_2}+\tilde{C}-(2l-e_1-e_2)|\\
&=|-(3;1^2)+(8;4,2)-(2;1^2)|=|(3;2,0)|
\end{align*}
and consider the exact sequence
$$0\rightarrow \h{O}(-\tilde{C}+\h{M})\rightarrow\h{O}(\h{M})\rightarrow\h{O}(\h{M})_{|\tilde{C}}\rightarrow 0 $$
Note that 
$\h{L}:=\h{O}(-\tilde{C}+\h{M})=\h{O}(-(5;2^2))$ and $\h{L}^{-1}=\h{O}((5;2^2))$ is ample (indeed very ample) and hence
$$h^1(\PP^2_5,\h{O}(-\tilde{C}+\h{M}))=0$$
by Kodaira vanishing theorem. 

By the surjectivity of restriction map 
$$H^0(\PP^2_2,\h{O}_{\PP^2_2}(K_{\PP^2_2}+\tilde{C}-(2,1^2))\rightarrow H^0(\tilde{C}, K_{\tilde{C}}(-1)),$$
$|K_{\tilde{C}}(-1)|$ is completely cut out on $\tilde{C}$ by $|K_{\PP^2_2}+\tilde{C}-(2l-e_1-e_2)|$. The morphism $Q_q\cong\PP^2_2\stackrel{|(3;2,0)|}{\hookrightarrow}\PP^6$ is a morphism contracting $e_2$ whereas $\w{C}\cdot e_2=2$. Hence $|K_{\w{C}}(-1)|$ is not very ample and so is the projection given by $|K_{\w{C}}(-1)-B|$ from the base point $B$. Note that $$|\h{O}_X(1)|=|K_X-K_X(-1)|=|K_{\w{C}}-(\h{O}_{\w{C}}(1)+B)|=|K_{\w{C}}(-1)-B|$$
via the identification  $X\cong\w{C}$. 

Assume $C$ lies on a quadric cone $Q$ with vertex $P$. Using the notation in Remark \ref{oo1}, let
$\mathbb{F}_2\longrightarrow Q\subset\PP^3$ be the morphism induced by $|h+2f|$. Let $\w{C}$ be the proper transformation of $C$ under the desingularization $\mathbb{F}_2\longrightarrow Q\subset\PP^3$ given by $|h+2f|$. Setting $\w{C}\in|ah+bf|$,
$\w{C}\cdot(h+2f)=(ah+bf)\cdot(h+2f)=b=10$, $0\le\w{C}\cdot h=-2a+b\le m$ where $m$ is the multiplicity of $C$ at the vertex $P$. Hence we have
\[
(a,b,m)=
\begin{cases}
(5,10,0) 
\\
(4,10,2). 
\end{cases}
\]
In the first case $(a,b,m)=(5,10,0)$, the ruling of the cone cut out a base point free $g^1_5$ which is impossible under the existence 
of a base point free $g^1_4$ by Casetelnuovo-Severi inequality.  In the case $(a,b,m)=(4,10,2)$, $p_a(\w{C})=15, \w{C}\in|4h+10f|$, and hence $\w{C}$ is singular. 
Choose $f_0$ the unique fibre containing the unique singular point $p_0\in\w{C}$. Blow up $\mathbb{F}_2$ at $p_0$ and let $e$ be the exceptional divisor of the blow up $\mathbb{F}_{2,1}\stackrel{\pi}{\rightarrow}\mathbb{F}_2$. Let $\tilde{f}_0$ ($\dbtilde{C}$ resp.) be the proper transformation of $f_0$ ($\w{C}$ resp.). By abusing notation, we denote 
$\pi^*(h) ~~\&~~ \pi^*(f)$ by $h ~~\&~~f$. We have
$$\dbtilde{C}\in|4h+10f-2e|$$
\begin{align*}
\h{M}&:=|\dbtilde{C}+K_{\mathbb{F}_{2,1}}-(h+2f)|\\&=|(4h+10f-2e)+(-2h-4f+e)-(h+2f)|\\&=|h+4f-e|
\end{align*}
 Since 
$\h{O}_{\mathbb{F}_{2,1}}(\h{M}-\dbtilde{C})=\h{O}_{\mathbb{F}_{2,1}}(-(3h+6f-e))$
and $f\cdot (\h{M}-\dbtilde{C})<0$, we see that $h^0({\mathbb{F}_{2,1}},\h{O}(\h{M}-\dbtilde{C}))=0$ implying that the restriction map 
$$\rho: H^0({\mathbb{F}_{2,1}},\h{O}(\h{M}))\longrightarrow H^0(\dbtilde{C},\h{O}(\h{M})\otimes\h{O}_{\dbtilde{C}})$$ is injective. 
Suppose $\rho$ is not surjective: 
$$\mathrm{Im}(\rho)\subsetneq H^0(\dbtilde{C},\h{M}\otimes\h{O}_{\dbtilde{C}}).$$
By the projection formula, $\pi_*\pi^*\h{O}_{\mathbb{F}_2}(h+4f))=\h{O}_{\mathbb{F}_2}(h+4f)$
 \begin{align*}
h^0(\mathbb{F}_{2,1},\pi^*\h{O}_{\mathbb{F}_2}(h+4f))&=h^0(\mathbb{F}_2, \pi_*\pi^*\h{O}_{\mathbb{F}_2}(h+4f)))\\&=h^0(\mathbb{F}_{2},\h{O}(h+4f))=8
\end{align*} 
(cf. \cite[Ex.2, p.53]{Beauville} for the last equality above) and hence
\begin{align*}h^0(\mathbb{F}_{2,1},\h{O}(\h{M}))&=h^0(\mathbb{F}_{2,1},\h{O}(h+4f-e))\\&=h^0(\mathbb{F}_{2},\h{O}(h+4f))-1=7.
\end{align*}
Then by the injectivity and the non-surjectivity of $\rho$, 
\begin{align*}r=\dim\PP(H^0(\dbtilde{C},\h{O}(\h{M})\otimes\h{O}_{\dbtilde{C}}))&>\dim\PP(H^0(\mathbb{F}_{2,1},\h{O}(\h{M}))\\&=\dim\PP(\mathrm{Im}(\rho))=6.
\end{align*}
Since $|\mathrm{Im}(\rho)|\subsetneq \PP(H^0(\dbtilde{C},\h{O}(\h{M})\otimes\h{O}_{\dbtilde{C}}))$ induces a morphism birational onto its image, the complete linear system $\PP(H^0(\dbtilde{C},\h{O}(\h{M})\otimes\h{O}_{\dbtilde{C}}))$ is still birationally very ample, which contradicts the Castelnuovo upper bound for the arithmetic genus of degree $\dbtilde{C}\cdot\h{M}=(4h+10f-2e)\cdot (h+4f-e)=16$ curves in $\PP^{r\ge 7}$; $\pi(16,7)=12<g=14$. Therefore we have 
$$\mathrm{Im}(\rho)= H^0(C,\h{O}(\h{M})\otimes\h{O}_{\dbtilde{C}})$$
and the restriction map $\rho$ is an isomorphism. 
Note that 
$$\tilde{f}_0^2=-1, \tilde{f}_0\cdot e=1,  h\cdot\tilde{f}_0=1,\dbtilde{C}\cdot \tilde{f}_0=2, 
\h{M}\cdot\tilde{f}_0=(h+4f-e)\cdot \tilde{f}_0=0.$$
Hence the morphism $\psi$ given by  $\h{M}=|\dbtilde{C}+K_{\mathbb{F}_{2,1}}-(h+2f)|$ contracts the $(-1)$ curve $\tilde{f}_0$ and the image $\psi (\dbtilde{C})\subset\PP^6$ acquires a singularity. Recall that $K_X(-1)=\h{O}_{\dbtilde{C}}(1)+B$. Hence 
\begin{align*}
\PP(H^0(\dbtilde{C},\h{O}(\h{M})\otimes\h{O}_{\dbtilde{C}}))&=|K_{\dbtilde{C}}(-1)|\\&=
|K_{\dbtilde{C}}-(K_X(-1)-B)|=|\h{O}_X(1)+B|.
\end{align*}
Since $\PP(H^0(\dbtilde{C},\h{O}(\h{M})\otimes\h{O}_{\dbtilde{C}}))=|\h{O}_X(1)+B|$ is not very ample, $\h{O}_X(1)$ is not very ample, a contradiction. 

\item[(ii)] $\deg B=0$ and $C$ is not contained in a quadric surface.


\quad (a) Suppose $K_X(-1)=g^3_{11}$ is very ample. Recall that a complete linear series $g^r_d$ on a $4$-gonal curve $C$ is of type 1 (type 2, resp.) if $g^r_d$ is composed of $g^1_4$ (if the residual series $|K_C-g^r_d|$ is composed of $g^1_4$, resp.) according to Coppens-Martens \cite{cm}. In the current situation, the linear series $K_X(-1)$ is neither type 1 nor type 2 and hence by \cite[Th. 1.9]{cm}, $K_X(-1)=g^3_{11}$ is of the form $|2g^1_4+F|$ with $\deg (F)=3$ i.e. there is $m\in \{0,2\}$
such that $g^3_{11} =|\frac{m+r-1}{2}g^1_4+F|$ and  $h^0(mg^1_4+F)=m+2$. Since we already avoided the case $C\in |\Oo_Q(4,7)|$ in Example \ref{ex1} (ii), $m\ne 0$ and therefore
$g^3_{11} =|2g^1_4+F|$. Write $F =p+F'$ with $\deg (F')=2$. Since $g^3_{11}=|2g^1_4+F|=|2g^1_4+F'+p|$ is very ample, we get $h^0(2g^1_4+p) = 2$, a contradiction.

\quad (b) Recall $p_a(C)\le 15$ by the assumption that $C$ is not contained in a quadric. Now assume $p_a(C)=15$. We exclude the case $g^3_{11} =|2g^1_4+F|$. Since $p_a(C)=15$ and $g=14$, $C$ can  have only one node or only one simple cusp (a double point). On the other hand,
since $g^3_{11} =|2g^1_4+F|$, $\deg F=3$, $F$ collapses under the morphism given by $|2g^1_4+F|$ onto a singular point of $C$ of multiplicity three which certainly cannot be a node or a simple cusp. 


\item[(iii)] Assume  $\deg B=0$ and $C$ contained in a smooth quadric $Q$.  Since we are assuming the existence 
of a $g^1_4$, $C\in|\h{O}_Q(5,6)|$ is impossible by the Casetelnuovo-Severi inequality.
In the case $C\in|\h{O}_Q(4.7)|$ we have $p_a(C)=18$ and
 $C$ may have rather bad singularities, not just nodes or cusps. Treating all the possible combinations 
of singularities on $C$ would be somewhat too much cumbersome. 
Instead, our strategy here is to show 
that a general curve in the Severi variety of curves of geometric genus $g=14$ on smooth $Q$ in the linear system $\h{L}=|\h{O}_Q(4.7)|$ does not have very ample $|K_{\w{C}}(-1)|$ where $\widetilde{C}$ is the normalization of  $C\subset Q\subset\PP^3$. Since being 
very ample is an open condition (Remark \ref{oo3}), this will show that there does not exist a smooth $4$-gonal curve $X\in\h{H}_{15,14,5}$ whose dual curve $C$ belongs to $|\h{O}_Q(4.7)|$, whatsoever the singularity of $C$ is.
On the other hand, we already showed that if $C\in|\h{O}_Q(4.7)|$ is general with four nodes, $\h{O}_X(1)=K_{\w{C}}(-1)$ is not very ample in Example \ref{ex1}(ii). 

\item[(iv)] Assume $\deg B=0$ and $C$ is contained in a quadric cone $Q$.
We consider the desingularization $\mathbb{F}_2\stackrel{\pi}{\longrightarrow} C\subset Q$ given by
$|h+2f|$.
Let $\w{C}\subset\mathbb{F}_2$ be the strict transformation of $C$ under $\pi$. Set
 $\w{C}\in|ah+bf|$ and we have
$\w{C}\cdot(h+2f)=(ah+bf)\cdot(h+2f)=b=11$, $0\le\w{C}\cdot h=-2a+b\le m$ where $m$ is the multiplicity of $C$ at the vertex $P$. 
Hence we have 
\[
(a,b,m)=
\begin{cases}
(5,11,1) 
\\
(4,11,3). 
\end{cases}
\]

The case $(5,11,1)$ is out of our consideration; the ruling of the cone cuts out a base point free $g^1_5$ which
is a contradiction while assuming the existence of $g^1_4$ by Castelnuovo-Severi inequality. 
If $\w{C}\in |4h+11f|$, by adjuction we have $p_a(\w{C})=18$. Again we adopt the same  strategy (as in the case $C\in|\h{O}_Q(4.7)|$ on a smooth quadric) to show 
that a general element in the Severi variety of curves of geometric genus $g=14$ on $\mathbb{F}_2$ in  $\h{L}=|4h+11f|$ does not have very ample $|K_{\dbtilde{C}}(-1)|$ where $\dbtilde{C}$ is the normalization of $\w{C}\subset\mathbb{F}_2$. The following is parallel to the case we already considered in the case $\deg B=1$, $(a,b,m)=(4,10,2)$ and $C$ lies on a quadric cone. However we provide some computations for the convenience of readers. 

We assume that $\w{C}\subset\mathbb{F}_2$ has four nodes as its only 
singularities in general position. Let $e_i (i=1,\cdots, 4)$ be exceptional divisors and let
$f_i (i=1,\cdots, 4)$ be fibers containing the four nodal points of $\w{C}$. After resolving
all the four nodes we get a smooth curve $\dbtilde{C}\subset\mathbb{F}_{2,4}$ on the Hirzebruch surface $\mathbb{F}_2$ blown up at four points.
We have
$$\dbtilde{C}\in|4h+11f-\sum2e_i|$$
\begin{align*}
\h{M}:&=|\dbtilde{C}+K_{\mathbb{F}_{2,4}}-(h+2f)|\\&=|(4h+11f-\sum2e_i)+(-2h-4f+\sum e_i)-(h+2f)|\\&=|h+5f-\sum e_i|
\end{align*}
Since 
$\h{O}_{\mathbb{F}_{2,4}}(\h{M}-\dbtilde{C})=\h{O}_{\mathbb{F}_{2,4}}(-(3h+6f-\sum e_i))$
and $f\cdot (\h{M}-\dbtilde{C})<0$, we see that $h^0({\mathbb{F}_{2,4}},\h{O}(\h{M}-\dbtilde{C}))=0$ implying that the restriction map 
$$\rho: H^0({\mathbb{F}_{2,4}},\h{O}(\h{M}))\longrightarrow H^0(\dbtilde{C},\h{O}(\h{M})\otimes\h{O}_{\dbtilde{C}})$$ is injective. 
Suppose 
$$\mathrm{Im}(\rho)\neq H^0(\dbtilde{C},\h{M}\otimes\h{O}_{\dbtilde{C}}).$$
Then 
$$\dim\PP(H^0({\mathbb{F}_{2,4}},\h{O}(\h{M})\otimes\h{O}_{\dbtilde{C}}))=\dim\PP(H^0(\dbtilde{C},\h{M}\otimes\h{O}_{\dbtilde{C}}))=r\ge 6.$$ 
Since $|\mathrm{Im}(\rho)|\subset \PP(H^0({\mathbb{F}_{2,4}},\h{O}(\h{M})\otimes\h{O}_{\dbtilde{C}}))$ induces a morphism birational onto its image, the complete linear system $\PP(H^0(\dbtilde{C},\h{O}(\h{M})\otimes\h{O}_{\dbtilde{C}}))$ is still birationally very ample, which contradicts the Castelnuovo upper bound for the arithmetic genus of degree $15$ curves in $\PP^{r\ge 6}$; $\pi(15,6)=13<g=14$. Therefore we have 
$$\mathrm{Im}(\rho)= H^0(C,\h{O}(\h{M})\otimes\h{O}_{\dbtilde{C}}).$$
Denoting by $\w{f}_i$  the proper transformation of $f_i$ under the blow up, 
$$\tilde{f}_i^2=-1, \tilde{f}_i\cdot e_i=1,  h\cdot\tilde{f}_i=1, \dbtilde{C}\cdot \tilde{f}_i=2, 
(h+5f-\Sigma e_i)\cdot \tilde{f}_i=0.$$
Hence the morphism $\psi$ given by $\h{M}=|\dbtilde{C}+K_{\mathbb{F}_{2,4}}-(h+2f)|$ contracts  $(-1)$ curves $\tilde{f}_i$ and the image curve $\psi (\dbtilde{C})\subset\PP^5$ 
acquires singularities.
Since the complete linear system $\mathrm{Im}(\rho)$ maps $\dbtilde{C}$ onto $\psi (\dbtilde{C})\subset\PP^5$ with at least $4$ singular points, 
$\h{M}\otimes\h{O}_{\dbtilde{C}}=|K_{\dbtilde{C}}(-1)|$ is not very ample.

\end{itemize}

\end{proof}

\section{Birationality of the moduli map}
Recall that  $\Gamma_1$ (resp. $\Gamma _2$) is the irreducible component of $\h{H}_{15,14,5}$ whose general element has a dual curve not contained (resp. contained) in a quadric surface.

Let $\mu _i: \Gamma_i\to \Mm_{14}$, $i=1,2$, denote the natural moduli map.
We first need the following easiest kind of super abundance lemma regarding configuration of points with respect to a certain linear system on a smooth quadric $Q\subset\PP^3$ and we omit its verification which is rather elementary. 
\begin{lemma}\label{aaa1}
Fix an integer $e$ such that $0\le e\le 6$ and  a set $A\subset Q$ such that $\#A=e$ and $h^1(\Ii_A(3,4)) \neq 0$. Then  either $e=6$ and there is $L_1\in |\Oo_Q(1,0)|$ such that $A\subset L_1$ or $e\in \{5,6\}$ and there is $L_2\in |\Oo_Q(0,1)|$ such that $\#(A\cap L_2)\ge 5$ and in the latter case there is $A'\subseteq A$
such that $\#A'=5$ and $h^1(\Ii_{A'}(3,4)) >0$.
\end{lemma}

\begin{lemma}\label{bo1}
Fix an integer $a$ such that $0\le a\le 6$. Let $G\subset Q$ be a general subset of $Q$ with cardinality $a$. Then $G$ is the set-theoretic  base locus of $|\Ii_G(2,2)|$.
\end{lemma}

\begin{proof}
It is sufficient to check  for $a=6$. Assume the existence of $p\in Q\setminus G$ such that
$|\Ii_G(2,2)| = |\Ii_{G\cup \{p\}}(2,2)|$. Since $G$ is general, no two points among $G$ lie on a line in $Q$. For any $E\subset G$ such that $\#E=3$ the non-empty linear system $|\Ii_E(1,1)|$  has a unique element $C_E$ which  is smooth.
The generality of $G$ implies that $h^0(Q,\Ii _B(1,1)) =0$
for all $B\subset G$ such that $\#B =4$ and hence $G\cap C_E=E$. Likewise, we let
$C_{G\setminus E}$ be the uniqe conic in $|\Ii_{G\setminus E}(1,1)|$. 
Since
$G\subset C_E\cup C_{G\setminus E}$, $C_E\cup C_{G\setminus E} \in |\Ii_G(2,2)| = |\Ii_{G\cup \{p\}}(2,2)|$ and  therefore $p\in C_E\cup C_{G\setminus E}$. Assume for instance $p\in C_E$. On smooth $C_E\cong \PP^1$, $\deg (\Oo_{C_E}(2,2)) =4$ and hence $\dim|\Oo_{C_E}(2,2)|=4$. Furthermore the pencil $|\Ii_{E,C_E}(2,2)|\subset |\Oo_{C_E}(2,2)|$ on $C_E$ has the base locus $E$.
Consider the exact sequence
\begin{equation}\label{eqbo1}
0\to \Ii_{G\setminus E}(1,1)\to \Ii_{G\cup \{p\}}(2,2)\to \Ii_{E\cup \{p\},C_E}(2,2)\to 0
\end{equation}
Since $p\in Q\setminus G$ is not in the base locus $E$ of the pencil $|\Ii_{E,C_E}(2,2)|$, we have $h^1(C_E,\Ii_{E\cup \{p\},C_E}(2,2))=0$. Since $G\setminus E$ is general and $\#(G\setminus E)=3$, we have
$h^1(Q,\Ii_{G\setminus E}(1,1)) =0$. Therefore it follows from \eqref{eqbo1} that $h^1(Q,\Ii _{G\cup \{p\}}(2,2)) =0$, contradicting the assumption that $p$ is in the base locus of $|\Ii_G(2,2)|$.
\end{proof}

\begin{lemma}\label{aaa2}
Fix an integer $a$ such that $0\le a\le 6$ and take a general $G\subset Q$ such that $\#G=a$. Fix a set $A\subset Q$ such that $\#A =6$, $G\cap A=\emptyset$, $h^1(Q,\Ii_{G\cup A}(3,4)) >0$ and
$h^1(Q,\Ii _{G\cup A'}(3,4)) =0$ for all $A'\subsetneq A$. Then there is $L\in |\Oo_Q(1,0)|$ such that $A\subset L$.
\end{lemma}

\begin{proof}
Since $h^1(\Ii _{G\cup A'}(3,4)) =0$ for all $A'\subsetneq A$, there is no $J\in |\Oo_Q(0,1)|$ such that $\#(A\cap J) \ge 5$. Thus the case $a=0$ is true by Lemma \ref{aaa1}. Thus we may assume $a>0$ and that the lemma is true for the integer $a-1$. Hence
either  there exists $L\in |\Oo_Q(1,0))|$ such that $A\subset L$ or $h^1(\Ii _{G'\cup A}(3,4))=0$ for all $G'\subsetneq G$. Thus from now on   we assume $h^1(\Ii_{G'\cup A}(3,4)) =0$ for all $G'\subsetneq G$. 
Fix $A'\subset A$ such that $\#A'=  5$. Set $\{p\}:=A\setminus A'$. Since $\#A'=5$, $|\Ii_{A'}(1,2)| \ne \emptyset$.
Take a general $T\in |\Ii_{A'}(1,2)|$. Consider the residual exact sequence of $T$: 

\begin{equation}\label{eqe0}
0\to \Ii_{G\cup  A\setminus (G\cup A)\cap T}(2,2)\to \Ii_{G\cup A}(3,4)\to \Ii_{T\cap (G\cup A),T}(3,4)\to 0
\end{equation}

\quad (a) First assume $p\notin T$. In this case we have $G\cup  A\setminus (G\cup A)\cap T \subseteq G\cup \{p\}$ and $T\cap (G\cup A')=T\cap (G\cup A)\subseteq G\cup A'$. Since $h^1(\Ii_{G\cup A'}(3,4)) =0$ by assumption, we have 
$h^1(T,\Ii_{T\cap (G\cup A),T}(3,4))=h^1(T,\Ii_{T\cap (G\cup A'),T}(3,4))=0$ from the exact sequence (\ref{eqe0}) adapted to the set $G\cup A'$. Hence by assumption $h^1(Q,\Ii_{G\cup A}(3,4))>0$, the long exact sequence on cohomology of the exact sequence of \eqref{eqe0} implies 
\begin{align*}h^1(Q,\Ii _{\{p\}\cup (G\setminus G\cap T)}(2,2))&=h^1(Q,\Ii _{G\cup A\setminus ((G\cup A)\cap T)}(2,2))\\&=h^1(Q,\Ii _{\{p\}\cup (G\cup A')\setminus ((G\cup A)\cap T)}(2,2)) >0.           
\end{align*}
 and hence $h^1(Q,\Ii_{G\cup \{p\}}(2,2)) >0$.
Since $G$ is general, $h^1(Q,\Ii_G(2,2)) =0$. Thus $p\notin G$ is in the base locus of $|\Ii_G(2,2)|$, contradicting Lemma \ref{bo1}. 

\bigskip

\quad (b) Now assume $p\in T$. In this case, $A\subset T$ and $G\cup  A\setminus (G\cup A)\cap T\subseteq G$. Since $G$ is general in $Q$, $\#G=a  \le h^0(\Ii_Q(2,2))=9$ we have 
$h^1(Q,\h{I}_G(2,2))=0$ and hence $h^1(Q,\Ii_{G\cup  A\setminus (G\cup A)\cap T}(2,2))=0$. Thus from \eqref{eqe0} and the assumption $h^1(Q,\Ii_{{{G}}\cup A}(3,4)) >0$ we have $h^1(T,\Ii_{T\cap (G\cup A),T}(3,4))>0$.  From the standard exact sequence of the restriction map,
\begin{equation*}\label{eqer}
0\rightarrow \h{I}_{G\cup A}(3,4)\rightarrow \h{I}_{{(G\cup A)}\cap T}(3,4)\rightarrow \h{I}_{{(G\cup A)}\cap T}(3,4)\otimes\h{O}_T\rightarrow 0
\end{equation*}
one has $h^1(Q,\Ii_{T\cap (G\cup A)}(3,4)) >0$ since
$h^1(T,\Ii_{T\cap (G\cup A),T}(3,4))>0.$
By the assumption $h^1(\Ii_{G'\cup A}(3,4)) =0$ for all $G'\subsetneq G$ and by
$h^1(Q,\Ii_{T\cap (G\cup A)}(3,4)) >0$, we have $T\cap G=G$ and $T\cap(G\cup A)=G\cup A\subset T$, hence $G\cup A$ is in the base locus $\Bb$ of $|\Ii_{A'}(1,2)|$. 
Therefore it follows that
$$0<h^0(Q,\Ii_{A'}(1,2))\le h^0(Q,\Ii_{G\cup A}(1,2))\le h^0(Q,\Ii_{G}(1,2))\le h^0(\Oo _Q(1,2)) -a.$$
For $\#G=6$, this is an obvious absurdity. 
 
 Now assume $\#G=a=5$.  In this case $|\Ii_G(1,2)|$ has a unique element  which is a twisted cubic $C\cong\PP^1$. Note that  $\deg \Ii_{G\cup A,C}(3,4)=\deg \Oo_C(3,4)-\deg \Ii_{G\cup A,C}=0$ and 
 we have  $h^1(C,\Ii_{G\cup A,C}(3,4)) =0$. 
 Since $h^1(Q,\Oo_Q(2,2)) =0$, the residual exact sequence \eqref{eqe0}  adapted to $C\subset Q$ gives a contradiction to the assumption $h^1(Q,\Ii_{G\cup A}(3,4)) >0$.
 \medskip
 
 (b1) Assume $a=4$.  Fix $A''\subset A$ such that $\#A''=4$ and take a general $T_1\in |\Ii_{A''}(1,2)|$. By step (a) applied to all $A'\subset A$ with $\#A'=5$, we may assume either $A\subset T_1$ or $A\cap T_1=A''$.

 \begin{itemize}
 \item[(b1.1)] First assume $A\subset T_1$ and hence $G\cup A \setminus (G\cup A)\cap T_1) \subseteq G$. 
Since $G$ is general, $h^1(Q,\Ii _G(2,2)) =0$ and hence $h^1(Q,\Ii_{G\cup A\setminus T_1\cap  (G\cup A)}(2,2)) =0$.
 
 Considering the long exact sequence on cohomology of \eqref{eqe0} with $T_1$ instead of $T$, we have $h^1(T_1,\Ii _{(G\cup A)\cap T_1}(3,4)) >0$. On the other hand from the 
long exact sequence on cohomology of the standard restriction map
 $$0\rightarrow \h{I}_{(G\cup A)}(3,4)\rightarrow\h{I}_{T_1\cap(G\cup A)}(3,4)\rightarrow \h{I}_{T_1\cap(G\cup A)}(3,4)\otimes\h{O}_{T_1}\rightarrow 0$$
 we have $h^1(Q,\Ii _{T_1\cap (G\cup A)}(3,4)) >0$. 
 By assumption
 $h^1(Q,\Ii _{G'\cup A}(3,4)) =0$ for all $G'\subsetneq G$, we have
$T_1\cap(G\cup A)=G\cup A$ and hence $T_1\supset G\cup A$. 
 Since $T_1$ is general in $|\Ii_{A''}(1,2)|$, $G\subset T_1$, $\#G=\#A''$ and $G$ is general, we get $$|\Ii_{A''}(1,2)|=|\Ii_{G\cup A}(1,2)| = |\Ii_G(1,2))|$$ and that each element of $|\Ii_G(1,2)|$ contains $A$. Since $\#G=4$ and $G$ is general the base locus of the pencil $|\Ii_G(1,2)|$ is the intersection of $2$
 general $C, C_1\in |\Oo_Q(1,2)|$. Since $\#(C\cap C_1) =4<10 =\#(G\cup A)$, we get a contradiction.

 \item[(b1.2)] Now assume $A\cap T_1=A''$.  Since $A'' \subsetneq A$, we have $h^1(Q,\Ii_{G\cup A''}(3,4)) =0$. Thus $h^1(T_1,\Ii _{T_1\cap (G\cup A'',T_1}(3,4)) =0$.
 Since $T_1\cap (G\cup A) =T_1\cap (G\cup A'')$, the long cohomology exact sequence of \eqref{eqe0}
 for $T_1$ gives 
 
$h^1(Q, G\cup A\setminus (G\cup A)\cap T_1)=h^1(Q, \Ii_{(G\setminus G\cap T_1)\cup (A\setminus A'')}(2,2)) >0$ and hence $h^1(Q,\Ii_{G\cup (A\setminus A'')}(2,2))>0$. 
 
 Take a general $D\in |\Ii_{A\setminus A''}(1,1)|$ and consider the exact sequence
 \begin{equation}\label{eqe01}
 0\to \Ii_{G\setminus D\cap G}(1,1)\to \Ii_{G\cup (A\setminus A'')}(2,2)\to \Ii_{D\cap (G\cup (A\setminus A''),D}(2,2)\to 0
 \end{equation}
 Since $\#(G\setminus G\cap D)\le 4$ and $G$ is general, $h^1(Q, \Ii_{G\setminus D\cap G}(1,1))=0$. Since $h^1(Q,\Ii_{G\cup (A\setminus A'')}(2,2))>0$, the long cohomology exact sequence of \eqref{eqe01} yields 
 $h^1(D, \Ii_{D\cap (G\cup (A\setminus A''),D}(2,2))>0$. Note that  $D$ is either a smooth rational curve or $D=L_1\cup L_2$ with
 $L_1\in |\Oo_Q(1,0)|$ and $L_2\in |\Oo_Q(0,1)|$. Since $\#G =4$, $h^0(Q,\Oo_Q(1,1)) =4$ and $G$ is general, we have $h^0(Q,\Ii _G(1,1)) =0$. Thus $G\nsubseteq D$. Thus $\#(D\cap (G\cup (A\setminus A''))\le 5$.
 
Assume that $D$ is a smooth rational curve. Since $\#(D\cap (G\cup (A\setminus A''))\le 5$ and by $$\deg \Ii_{D\cap (G\cup (A\setminus A''),D}(2,2)\ge -5+\deg(\Oo_D(2,2)) =-1$$
 we get  $h^1(D, \Ii_{D\cap (G\cup (A\setminus A''),D}(2,2))=0$
 , a contradiction.
 
 Now assume $D=L_1\cup L_2$. Since $G$ is general and $h^0(\Oo_Q(1,0)) =h^0(\Oo_Q(0,1)) =2$, $\#(L_i\cap G)\le 1$ for all $i$ and hence
 $\#(D\cap G)\le 2$ and $\#(D\cap G)=1$ if $L_1\cap L_2\in G$. Recall that $h^1(D, \Ii_{D\cap (G\cup (A\setminus A''),D}(2,2))>0$  and hence $h^1(Q, \Ii_{D\cap (G\cup (A\setminus A'')}(2,2))>0$. Set $W:=D\cap (G\cup (A\setminus A''))$. Fix $i\in \{1,2\}$
 such that $\#(W\cap L_i)\ge \#(W\cap L_{3-i})$. Consider the residual exact sequence of $L_i$:
 \begin{equation}\label{eqe02}
 0\to \Ii_{W\setminus W\cap L_i}(2,2)(-L_i)\to \Ii_W(2,2)\to \Ii_{L_i\cap W,L_i}(2,2)\to 0
 \end{equation}
Since $\#(G\cap L_i)\le 1$ and $\#(A\setminus A'')=2$, we have $\#(W\cap L_i)\le 3$. Since $\deg (\Oo_{L_i}(2,2)) =2$, we have $h^1(L_i,\Ii_{L_i\cap W,L_i}(2,2))=0$. Thus from the long cohomology exact sequence of \eqref{eqe02}, 
$h^1(Q, \Ii _{W\setminus W\cap L_i}(2,2)(-L_i)) >0$. We have $\#W\le 5$ and $\#(L_i\cap W)\ge \#(L_{3-i}\cap W)$, we have $\#(W\setminus W\cap L_i)\le 2$. Thus 
$h^1(Q,\Ii _{W\setminus W\cap L_i}(2,2)(-L_i))=0$, a contradiction. 
 \end{itemize}
  
\quad (b2) Assume $1\le a\le 3$. Fix $A_1\subset A$ such that $\#A_1=3$. Take a general $D_1\in |\Ii_{A_1}(1,1)|$. We use \eqref{eqe01} for $D_1$ instead of $D$. 

\noindent
If $h^1(Q, \Ii_{G\cup A\setminus (D_1\cap (G\cup A))}(2,2)) =0$, we conclude as above.
Thus we may assume $h^1(D,\Ii_{D_1\cap (G\cup A)}(2,2))>0$. We conclude as in in the last part of step (b1.2).
\end{proof}

\begin{lemma}\label{e6.00}
For all integers $0\le e\le 6$, the normalization $Y$ of a general nodal $D\in |\Oo_Q(5,6)|$ with exactly $e$ nodes has a unique base-point free $g^1_6$, the one induced by $|\Oo_Q(1,0)|$, and a unique $g^1_5$, the one induced by $|\Oo_Q(0,1)|$.

\end{lemma}

\begin{proof}
The curve $D$ has geometric genus $g=20-e$. We first prove that $Y$ has a unique $g^1_5$.
First of all $Y$ has no $g^1_4$, because any element of $|\Oo_{\PP^1\times \PP^1}(5,4)|$ has arithmetic genus $12<20-e$. Assume that $Y$ has another base point free $g^1_5$.
Since $5$ is prime, these two base point free $g^1_5$ induces a morphism $\pi : Y\to \PP^1\times \PP^1$ birational onto its image $\pi(Y)\in |\Oo_{\PP^1\times \PP^1}(5,5)|$. Since
$\pi(Y)$ has arithmetic genus $q=16$, we have $4\le e\le 6$. 

The set of all possible such $\pi(Y)$'s on a fixed $Q$ depend on at most $$\dim |\Oo_{\PP^1\times \PP^1}(5,5)|  -(q-g)=35-(16-(20-e)) =39-e$$ parameters by Remark \ref{oo1}.
Since the set of all $D$ is irreducible  dimension $41-e> 39-e$ a curve with two $g^1_5$ is not general in $\w{\Sigma}_g$, i.e does not constitute a component.

Now we prove that a general $Y$ has a unique base point free $g^1_6$, the one induced by $|\Oo_Q(1,0)|$.
Set $E:= \mathrm{Sing}(D)$. Since $D$ is general, $E$ is a general subset of $Q$ with cardinality $e$; Remark \ref{oo1}.
Let $f: Y\to D$ be the normalization map. Fix a possibly incomplete base point free pencil $\Ll\in \mathrm{Pic}^6(Y)$ whi. We need to prove that $\Ll \cong f^\ast (\Oo_D(1,0))$. Since $Y$ has a unique $g^1_5$, $h^0(\Ll)=2$. For any $G\subseteq E$ let $f_G: Y_G\to D$ be the partial normalization of $D$ in which we only normalize the points of $G$. Hence $Y_G$ is nodal with exactly $e-\#G$ nodes. Let $a_G: Y\to Y_G$ denote the normalization map. Take a minimal $G\subseteq E$ such that there is a degree $6$ line bundle $\Rr$
on $Y_G$ such that $h^0(\Rr)\ge 2$ and $a_G^\ast(\Rr) =\Ll$. Since $h^0(\Ll)=2$ and $\Ll$ is base point free, $h^0(\Rr)=2$ and $\Rr$ is base point free. Fix a general $B\in |\Rr|$ and set $A:= f_G(B)\subset D$. Since $\Rr$ is base point free, $B$ is formed by $6$ distinct points of $Y_G$ and $G\cap B=\emptyset$.
Since $h^0(\Rr) =2$,  Riemann-Roch on $Y_G$ gives $B$ imposes exactly $5$ independent conditions on $|\omega _{Y_G}|$. Since the canonical series $|\omega_{Y_G}|$ is completely cut out on $D$ by $|\omega _Q+D|=|\Oo_Q(-2,-2)+ \Oo_Q(5,6)|=|\Oo_Q(3,4)|$,
we get $h^1(Q,\Ii_{G\cup A}(3,4)) >0$. Since $h^0(\Rr)=2$ and $\Rr$ is base point free, $h^1(Q,\Ii_{G\cup A'}(3,4)) =0$ for all $A'\subsetneq A$.
 Thus Lemma \ref{aaa2} concludes the proof.
 \end{proof}

\begin{proposition}\label{e6}
The generic fiber of the map $\mu _2: \Gamma_2\to \Mm_{14}$ is formed by a unique orbit by the group $\mathrm{Aut}(\PP^5)$, i.e. a general $X\in \Gamma_2$ has a unique  $g^5_{15}$.
\end{proposition}

\begin{proof}
A non-empty open subset $W$ of $\Gamma _2$ is formed by the normalizations of all nodal $C\in |\Oo_Q(5,6)|$ in which the set $E$ of the nodes is `` general '' in the sense that each for all $a, b\in \NN$ and each $S\subset E$ we have
$h^0(Q,\Ii_S(a,b)) =\max \{0,(a+1)(b+1)-\#S\}$. Fix a general $X\in \Gamma_2$. In particular $X\in W$. Assume that $X$ is isomorphic to another element $X_1\in \Gamma_2$. To prove the proposition we need to prove that $X$ and $X_1$ are projectively equivalent. 
For a general $X$ the curve $X_1$ is general in $W$. Let $C_1$ be dual curve of $X_1$. It is sufficient to prove that $C$ and $C_1$ are projectively equivalent.
Since $C_1$ is a general nodal element of $|\Oo_Q(5,6)|$ of geometric genus $14$ and $X\cong X_1$, the case $e=6$
of Lemma \ref{e6.00} applied to $C_1$ and $C$ give that both $C_1$ and $C$ are obtained by the unique pair of base point free $g^1_5$ and $g^1_6$ on $X\cong X_1$ so that $\h{O}_{C_1}(1)\cong\h{O}_C(1)=|g^1_5+g^1_6|$, hence $C$ and $C_1$ are projectively equivalent. 
\end{proof}

\begin{remark}\label{e31}
By Proposition \ref{b1} and Proposition \ref{a1}, all smooth $X\in\Gamma_2$ 
is $5$-gonal.

\end{remark}

\begin{proposition}\label{e21}
Fix a general $X\in \Gamma _1$. Then

\quad (a) $X$ has a unique very ample $g^5_{15}$, $\Oo_X(1)$, and a unique very ample $g^3_{11}$, $K_X(-1)$.

\quad (b) The general fiber of the morphism $\mu_1: \Gamma _1\to \Mm_{14}$ has a unique orbit under the action of the group $\mathrm{Aut}(\PP^5)$.
\end{proposition}

\begin{proof}
Note that part (b) is just a translation of the uniqueness of the very ample $g^5_{15}$ claimed in part (a).

We proved that the set of all $X_1\in \Gamma_1$ whose dual is not smooth have lower dimension. In particular we may assume that the dual curve $C$ of a general $X\in \Gamma_1$ is smooth. We may also assume;  $C$ is  not in a cubic surface (\cite[Prop. B1]{gp1}, $C$ is ACM, $C$ is $7$-gonal and that for each $g^1_7$, $|R|$, on $C$ there is a $4$-secant line $J$ of $C$ such that $|R|$ is induced by the pencil of planes through $J$ (Proposition \ref{e2})..

To prove part (a) by duality it is sufficient to prove that the dual $C$ of  a general $X\in\Gamma_1$
has a unique very ample $g^3_{11}=|\Oo_C(1)|$.  
Take any very ample $|M|= g^3_{11}$ on $C$. Call $C_1\subset \PP^3$ the image of $C$ by this $|M|$, i.e. call $\phi: C\to C_1\subset \PP^3$ the embedding associated to this $|M|$. By assumption $\phi$ is an isomorphism of genus $14$ smooth curves. Note that $M= \phi^\ast (\Oo_{C_1}(1))$. Two line bundles on $C$ are isomorphic if they are associated to the same effective divisor of $C$. Thus to prove that $|M| =|\Oo_C(1)|$, i.e. to prove that $M\cong \Oo_C(1)$ it is sufficient to prove the existence of $A\in |M|$ with $A\in |\Oo_C(1)|$, i.e. sufficient to find a plane $H\subset \PP^3$
such that $\phi^{-1}(H\cap C_1)\in |\Oo_C(1)|$. The generality of $X$ \cite[Prop. B1]{gp1} imply that $C_1$ is not contained in a cubic surface. Thus $C_1$ is ACM and in the irreducible component $\Delta$ of the set
of all ACM space curves consisting of those which are linked to curves of genus $2$ and degree $5$. Since $X$ is general in $\Gamma_1$, $C_1$ is general in $\Delta$ and hence its gonality is computed by $4$-secant lines, i.e. $C_1$ has gonality $7$ and for
each base point free $g^1_7= |R_1|$ on $C_1$ there is a $4$-secant line $J_1$ such that for all $E_1\in |R_1|$ there is a plane $H\supset J_1$ such that $H\cap C_1=(J_1\cap C_1)+E_1$. Fix $E_1\in |R_1|$ and set $E:= \phi^{-1}(E_1)$. Since $\phi$ is an isomorphism,
$|R|:= |E|$ is a also $g^1_7$. Thus there is a $4$-secant line $J$ of $C$ such that $|R|$ is induced by the intersection with $C$ of all planes containing $J$. Take a plane  $H\supset E$. We have $H\cap C\in |\Oo_C(1)|$ and $H\cap C = (C\cap J)+E$.
Note that $C\cap J$ is the base locus of $\Oo_C(1)(-E)$. Since $E =\phi^{-1}(E_1)$ and $J_1\cap C_1$ is the base locus of $|\Oo_{C_1}(1)(-E_1)|$ and  $\phi^{-1}(J_1\cap C_1) = J\cap C$. Thus $\phi^{-1}((J_1\cap C_1)+E_1)) = (J\cap C)+E\in |\Oo_C(1)|$.
Thus $\phi^\ast (\Oo_{C_1}(1)) \cong \Oo_C(1)$.
\end{proof}

\section{An epilogue, a small remark on the irreducibility of $\h{H}_{g+1,g,5}$}

So far we have dealt with a special case ($g=14$) of the Hilbert scheme $\h{H}_{g+1,g,5}$
and showed its reducibility. 
For $g\le 13$, the followings are known already.
\begin{remark}\label{lowg}
\begin{itemize}
\item[(i)]
$\h{H}_{g+1,g,5}\neq \emptyset$ if and only if $g\ge 12$. 
\item[(ii)]
For $g=12$, $\h{H}_{g+1,g,5}$ is reducible consisting of curves of maximal genus and the reducibility is known by \cite[Corollary 3.12, page 92]{he} \cite[proof of Theorem 2.4]{speciality}.
\item[(iii)]
For $g=13$, $\HL{g+1,g,5}$ is irreducible by \cite[Theorem 3.4]{lengthy}. 

\noindent Since $\pi(14,13,6)<g=13$  
 we have $\HL{g+1,g,5}=\h{H}_{g+1,g,5}$ and is irreducible.
\end{itemize}
\end{remark}

The case $g=14$ which we treated in this paper is the first non-trivial case in this context.
For curves with higher $g\ge 15$, virtually nothing is known about the 
irreducibility of $\h{H}_{g+1,g,5}$. However the main result of this paper suggests that 
the irreducibility of $\h{H}_{d,g,5}$ beyond the range $d\ge g+5$ conjectured by Severi does not 
hold  for $d$ not too much below $g+5$. 

\medskip
\bigskip
\noindent
 {\bf \Large{Declarations}}

\medskip
\noindent
{\bf\large{Conflict of interest}} The authors have no conflict of interest.


\begin{thebibliography}{99}
\bibitem{Accola1}
{R. Accola},
{\it Topics in the theory of Riemann surfaces.}
Lecture Notes in Mathematics 1595, Springer, Heidelberg, 1991.

\bibitem{ACGH}
{E. Arbarello, M. Cornalba, P. Griffiths and J. Harris},
\textit{Geometry of Algebraic Curves Vol.I.}
Springer-Verlag, Berlin/Heidelberg/New York/Tokyo, 1985.
\bibitem{ACGH2}
{E. Arbarello, M. Cornalba and  P. Griffiths},
\textit{Geometry of Algebraic Curves Vol.II.}
Springer, Heidelberg, 2011.
\bibitem{JPAA}
{E. Ballico, C. Fontanari and C. Keem}, 
\textit{On the Hilbert scheme of linearly normal curves in $\mathbb{P}^r$ of relatively high degree.}
J. Pure Appl. Algebra (2020) \textbf{224} (2020), 1115--1123.
\bibitem{Beauville}
{A. Beauville},
\textit{Complex algebraic surfaces.}
Cambridge University Press, London/New York, 1983.
\bibitem{Brevik}
{J. Brevik},
\textit{Curves on normal rational cubic surfaces}, Pacific J. Math. \textbf{230} (2007), no. 1, 73--105.
\bibitem{cm} M. Coppens and G. Martens, {\it Linear series on 4-gonal curves}, Math. Nachr. 213 (2000),
\bibitem{dedse} T. Dedieu and E. Sernesi, {\it Equigeneric and equisingular families of curves on surfaces}, Publ. Mat. 61 (2017), 175--212.
\bibitem{sandra}
{S. Di Rocco},
\textit{$k$-very ample line bundles on Del-Pezzo surfaces}.
Math. Nachr. , \textbf{179} (1996), 47--56.
\bibitem{E1}
{L. Ein},
\textit{Hilbert scheme of smooth space curves}.
Ann. Scient. Ec. Norm. Sup. (4), \textbf{19} (1986), no. 4, 469--478.
\bibitem{E2}
{L. Ein},
{\it The irreducibility of the Hilbert scheme of complex space curves.}
Algebraic geometry, Bowdoin, 1985 (Brunswick, Maine, 1985), Proc. Sympos. Pure Math., 46, Part 1, Providence, RI: Amer. Math. Soc., 83--87.
\bibitem{gp1}
Gruson, Laurent; Peskine, Christian, {\it Genre des courbes de l’espace projectif II},
Annales scientifiques de l’\'{E}.N.S. 4e s\'{e}rie, tome 15, no 3 (1982), p. 401--418
\bibitem{gp2}
Gruson, Laurent; Peskine, Christian
{\it Section plane d'une courbe gauche: postulation.(French)[Plane section of a space curve: postulation]}, Enumerative geometry and classical algebraic geometry (Nice, 1981), pp. 33–35,
Progr. Math., 24,
Birkh\"{a}user, Boston, Mass., 1982


\bibitem{he} J.~Harris, {\it Curves in projective space}, S\'eminaire de Math\'ematiques Sup\'erieures, vol.~85, Presses de
l'Universit\'e de Montr\'eal, Montreal, Que., 1982, With the collaboration of D. Eisenbud. 
\bibitem{hs} R. Hartshorne and E. Schlesinger, {\it Gonality of a general ACM curve in $\PP^3$,} Pacific J. Math. 251 (2011), no. 2, 269--313.
\bibitem{Hartshorne}
{R. Hartshorne},
\textit{Algebraic Geometry.}
Springer-Verlag, Berlin/Heidelberg/New York, 1977.
\bibitem{Zeuthen}
{R. Hartshorne}, \textit{Families of curves in $\PP^3$ and Zeuthen's problem.} Memoirs of the Americal Mathematical Society, \textbf{130}, no. 617, (1997).
\bibitem{I2}
{H. Iliev},
\textit{On the irreducibility of the {H}ilbert scheme of curves in
              {$\Bbb P^5$.}} Comm. Algebra., \textbf{36} (2008), no. 4, 1550--1564.
\bibitem{lengthy}
{C. Keem},
\textit{Existence and the reducibility of the Hilbert scheme of linearly normal
  curves in $\mathbb{P}^r$ of relatively high degrees,} J. Pure Appl. Algebra \textbf{227} (2023), 1115--1123.
\bibitem{speciality}
{C. Keem},
\textit{On the Hilbert scheme of  linearly normal curves in $\mathbb{P}^r$ with small index of speciality.} Indag. Math.(N.S.), \textbf{33} (2022), 1102--1124.
\bibitem{KKy1}
{C. Keem and Y.-H. Kim},
\textit{Irreducibility of the Hilbert Scheme of smooth curves in $\PP^3$ of degree $g$ and genus $g$.} 
Arch. Math., \textbf{108} (2017), no. 6, 593--600.
\bibitem{KKy2}
{C. Keem and Y.-H. Kim},
\textit{Irreducibility of the Hilbert Scheme of smooth curves in $\PP^4$ of degree $g+2$ and genus $g$.} 
\bibitem{KK3}
{C. Keem and Y.-H. Kim},
\textit{On the Hilbert scheme of linearly normal curves in $\mathbb{P}^4$ of degree $d = g+1$ and genus $g$.}
Arch. Math., \textbf{113} (2019), no. 4, 373--384.
\bibitem{ms} G. Martens and F.-O. Schreyer, {\it Line bundles and sygygies of trigonal curves,} Abh. Math. Sem. Univ. Hamburg 56
(1985), 169--189.
\bibitem{Sakai} 
    {Ohkouchi, Masahito and Sakai, Fumio},
     {\it The gonality of singular plane curves},
    {Tokyo J. Math.},
    {27},
     {2004},
    {1},
     {137--147}.
\bibitem{ps}  C. Peskine and L. Szpiro, {\it Liaison des vari\'{e}t\'{e}s alg\'{e}briques I}, Invent. Math. 26 (1974), 271--302. 


      \bibitem{Sev}
{F.  Severi}, 
\textit{Vorlesungen \"uber algebraische Geometrie.}
Teubner, Leipzig, 1921.






 \bibitem{ty} I. Tyomkin, {\it On Severi Varieties on Hirzebruch surfaces}, International Mathematics Research Notices, Vol. 2007.

\end{thebibliography}
\end{document}